\documentclass[preprint,12pt]{GU}
\usepackage{stmaryrd}
\usepackage{amsfonts}
\usepackage{mathrsfs}
\usepackage{bbm}
\usepackage{lineno,hyperref}
\usepackage{latexsym}
\usepackage[text={180mm,254mm},centering]{geometry}
\usepackage{latexsym,amssymb,amsmath,amsbsy,amsopn,amstext,color,multicol,amsthm}

\usepackage{bm}
\usepackage{extarrows}
\usepackage{amsfonts, stmaryrd}

\usepackage{mhequ}

\newtheorem{thm}{Theorem} 
\newtheorem{defn}{Definition} 
\newtheorem{lem}{Lemma} 
\newtheorem{prop}{Proposition}

\usepackage{chngcntr}
\counterwithin{thm}{section}
\counterwithin{defn}{section}
\counterwithin{lem}{section}
\counterwithin{prop}{section}

\newcommand{\be}{\begin{equs}}
\newcommand{\ee}{\end{equs}}
\newcommand{\dist}{\ensuremath{{\rm dist}}}

\newcommand{\loc}{\ensuremath{{\rm loc}}}

\counterwithin{equation}{section}

\allowdisplaybreaks[4] 

\begin{document}
\begin{frontmatter}
\title{Pullback $V$-Attractors of the Stochastic Calmed 3$D$ Navier-Stokes Equations\tnoteref{label1}}

\tnotetext[label1]{This work was partially supported by NSFC Grant No.12271443.}

\author{Yawen Duan}
\ead{dywen905@163.com}

\author{Anhui Gu\corref{cor1}}
\ead{gahui@swu.edu.cn}
\address{School of Mathematics and Statistics, Southwest
University, Chongqing 400715, China}

\cortext[cor1]{Corresponding author.}

\begin{abstract}
In this paper, we investigate a calmed version of the 3$D$ rotational Navier-Stokes equations driven by additive noise. First, we use the Ornstein-Uhlenbeck process to transform the equation into a random one. By using the Galerkin approximation, we establish the global well-posedness of solutions for the calmed system. Then, we demonstrate the existence of a closed, measurable $\mathcal{D}_V$-pullback absorbing set. Finally, by proving the pullback flattening property, we obtain the existence of a $\mathcal{D}_V$-pullback attractor in \(V\).
\end{abstract}


\end{frontmatter}

\section{Introduction}\label{sect1}
In this paper, we focus on the three-dimensional incompressible constant-density calmed rotational Navier-Stokes equations with additive noise on a bounded domain $\mathcal{O}\subset\mathbb{R}^3$: for given \(\tau\in\mathbb{R}\), \(T>\tau\),
\begin{align} \label{calmed u}
	\begin{cases}
		\frac{d}{dt}u+(\nabla \times u)\times \zeta^{\epsilon}(u)+\nabla \pi = \nu \triangle u+f+h(x)\frac{dW(t)}{dt} \ &\text{in} \ \mathcal{O}\times (\tau,T), \\
		\nabla \cdot u = 0 \ &\text{in} \ \mathcal{O} \times(\tau,T), \\
		u|_{\partial \mathcal{O}} = 0 \ &\text{on} \  \partial \mathcal{O} \times(\tau,T), \\
		u(x,\tau) = u_0 (x) \ &\text{in} \ \mathcal{O}.
	\end{cases}
\end{align}
Here, $ \nu >0$ is the kinematic viscosity; $u: \mathcal{O} \times [\tau,T] \to \mathbb{R} ^3$ represents the fluid velocity; $f: \mathcal{O} \times [\tau,T] \to \mathbb{R} ^3$ is the body force; $p: \mathcal{O} \times [\tau,T] \to \mathbb{R}$ is the pressure; $\pi:=p+\frac{1}{2}|u|^2$ denotes the dynamic pressure; $W(t)$ is an infinite-dimensional Wiener process; $\zeta^\epsilon$ is a calming function as defined in Definition~\ref{def:calming function}, and $h(x)\in \mathcal{D}(A)$  (with \(A\) being the Stokes operator, see Section 2.1 below) represents a deterministic function.

We observe that when $\zeta^\epsilon(u)\equiv u$, equation \eqref{calmed u} reduces to a classical stochastic Navier-Stokes equation in its rotational form given by:
\begin{align} \label{original}
	\begin{cases}
		\frac{d}{dt}u+(\nabla\times u)\times u+\nabla \pi = \nu \triangle u+f+h(x)\frac{dW(t)}{dt} \ &\text{in} \ \mathcal{O} \times (\tau,T), \\
		\nabla \cdot u = 0 \ &\text{in} \ \mathcal{O}\times(\tau,T), \\
		u|_{\partial \mathcal{O}} = 0 \ &\text{on} \  \partial \mathcal{O} \times(\tau,T), \\
		u(x,\tau) = u_0 (x) \ &\text{in} \ \mathcal{O}.
	\end{cases}
\end{align}

In this work, we use a novel modification termed `calming', achieving controlled attenuation of advective transport via smooth nonlinear truncation. The concept of calming was first introduced by the authors in \cite{enlow2024algebraic} for the 2D Kuramoto-Sivashinsky equations and then subsequently used to deal with the 3D Navier-Stokes equations \cite{enlow2024calmedNS} and the 2D Magnetohydrodynamic-Boussinesq equations \cite{enlow2024calmedOH}.
The method works by modifying the convective velocity field through an energy-dependent damping operator that preserves the original boundary conditions. Actually, applying a bounded truncation operator to the nonlinear term in 3D Navier-Stokes equations was also considered by Yoshida and Giga (1984) and Caraballo et at. (2006) in the study of the globally modified Navier-Stokes equations (see also \cite{Caraballo2023Random,caraballo2013three,2006Unique,hang2024random,kloeden2007pullback,romito2009uniqueness}). Importantly, compared to the above conventional smoothing and filtering approaches, calming avoids artificial viscosity injection and has its own advantages (see e.g. \cite{enlow2024calmedNS}).

From \cite{enlow2024calmedNS}, the following different forms of calming functions are considered, namely,
\begin{align}
	\zeta^\epsilon (x)=	\begin{cases}
		\zeta _1 ^\epsilon (x):=\frac{x}{1+\epsilon |x|}, &\text{or}\\
		\zeta _2 ^\epsilon (x):=\frac{x}{1+\epsilon^2 |x|^2}, &\text{or} \\
		\zeta _3 ^\epsilon (x):=\frac{1}{\epsilon}\arctan(\epsilon x), &\text{or}\\
		\zeta _4 ^\epsilon (x):=q^\epsilon(|x|)\frac{x}{|x|},
	\end{cases}\notag
\end{align}
where the arctangent in $\zeta_3 ^\epsilon$ acts component-wise:
$$\arctan\left((z_1,z_2,z_3)^T\right)=(\arctan(z_1),\arctan(z_2),\arctan(z_3))^T,$$
and
\begin{align}
	q^\epsilon(r)=\begin{cases}
		r,&0\leq r<\frac{1}{\epsilon},\notag\\
		-\frac{2}{\epsilon}(r-\frac{2}{\epsilon})^2+\frac{3}{2\epsilon},&\frac{1}{\epsilon}\leq r<\frac{2}{\epsilon},\notag\\
		\frac{3}{2\epsilon},&r\geq\frac{2}{\epsilon}.\notag
	\end{cases}
\end{align}
Note that $\zeta^\epsilon(x) \to x$ for all $x\in \mathcal{O}$ as $\epsilon \to 0$, and that $\zeta_i^\epsilon \in C^1$ for $i=1,2,3,4$. We further require that $\zeta^\epsilon$ satisfies the conditions specified in the following definition (see, e.g., \cite{enlow2024calmedNS}).
\begin{defn}\label{def:calming function}
	If the following three conditions are satisfied, we call that $\zeta^\epsilon:\mathbb{R}^3\to\mathbb{R}^3$ is a calming function:
	\item[(1)] $\zeta^\epsilon$ is Lipschitz continuous with Lipschitz constant 1.
	\item[(2)] For each fixed $\epsilon>0$, $\zeta^\epsilon$ is bounded.
	\item[(3)] There exist constants $C > 0$, $\alpha > 0$, and $\beta\geq 1$ such that for every $x\in\mathbb{R}^3$,
	\begin{align}
		|\zeta^\epsilon(x)-x|\leq C\epsilon^\alpha|x|^\beta.\notag
	\end{align}
\end{defn}
This paper first establishes the existence and uniqueness of solutions to \eqref{calmed u}. Central to our approach is the use of the Ornstein-Uhlenbeck process to transform the original equation into a random differential equation with random coefficients. This transformation enables the application of classical deterministic techniques. The key insight is that employing a calming function within the nonlinear term provides otherwise unattainable control over the \(L^\infty\) norm, thereby avoiding the need to alter boundary conditions or introduce higher-order derivatives \cite{enlow2024calmedNS}.
Next, we proceed by defining a continuous cocycle $\Phi$ through the solution. Subsequently, we establish that this cocycle satisfies the conditions required for the existence of a $\mathcal{D}$-pullback attractor in the space \(V\). The primary challenge in proving the existence of the pullback attractor in \(V\) arises when demonstrating asymptotic compactness via standard techniques. Such approaches typically necessitate higher-order estimates in the \(H^2\) space, coupled with the application of compact embedding theorems, inevitably involving complex inequalities or norms with higher regularity. To circumvent this difficulty, we adopt the flattening method introduced in \cite{kloeden2007flattening}.
This flattening technique bypasses the standard compactness obstacle by employing finite-dimensional projections. It requires only \(H_0^1\) energy estimates, avoiding higher regularity. In addition, the framework of pullback flattening, which is universal and applicable to both random and deterministic (autonomous or nonautonomous) dynamical systems, is relatively straightforward to verify, as its estimates are similar to those for bounded absorbing sets. Therefore, proving the pullback flattening property for $\Phi$ suffices to guarantee the existence of pullback attractors in \(V\) (see e.g.\cite{kloeden2007flattening}).

This paper is organized as follows. In section \ref{sect2}, we present some preliminaries on pullback attractors as well as some assumptions.
In Section \ref{sect3}, we establish the global well-posedness of \eqref{calmed u} in \(V\) by using the classical Galerkin method. Finally in section \ref{sect4}, we prove the existence of a $\mathcal{D}_V$-pullback attractor for the calmed system \eqref{calmed u}.

Throughout this paper, the constant \(C_{\alpha_1,\alpha_2,...,\alpha_N}\) depending on different parameters \(\alpha_1,\alpha_2,...,\alpha_N\), may change from line to line. The notations \(c_1,c_2,...,c_N\) represent generic positive constants.

\section{Preliminaries}\label{sect2}
\subsection{Notation and hypotheses}
We use standard notation borrowed from \cite{constantin1988navier,robinson2001infinite,temam2001navier}. Let $\mathcal{O}\subset\mathbb{R}^3$ be a bounded, open, connected convex set with $C^2$ boundary. Then there  exist positive constants $c_1$ and $c_2$, such that
\begin{align}\label{AuL2 uH2}
	c_1\|Au\|_{L^2}\leq\|u\|_{H^2}\leq c_2\|Au\|_{L^2},
\end{align}
where $A$ is defined in \eqref{A def} (cf. \cite{dauge1989stationary,guermond2011note}).
Let $C_{c}^{\infty}(\mathcal{O})$ denote the space of smooth test functions with compact support from $\mathcal{O}$ to $\mathbb{R}^3$. Then, $H_0^1(\mathcal{O})\equiv H_0^1$ is defined as the closure of $C_{c}^{\infty}(\mathcal{O})$ in $H_0^1(\mathcal{O})$.
Define
\begin{align}
	\mathcal{V}=\left\{\phi\in C_{c}^{\infty}(\mathcal{O}):\nabla\cdot\phi=0 \right\},\notag
\end{align}
and let $H$ and $V$ be the closures of $\mathcal{V}$ through $L^2(\mathcal{O})$ and $H_0^1(\mathcal{O})$ respectively.
We also let
\be
	(u,v):=\sum_{i=1}^{3}\int_{\mathcal{O}}u_i(x)v_i(x)dx, \quad
		\|u\|_{H_m}:=\biggl(\sum_{|\alpha|\leq m} \|D^\alpha u\|_{L^2}^2\biggr)^{\frac{1}{2}},
\ee
denote the $L^2$ inner product and the $H^m$ Sobolev norm. Let $\alpha=(\alpha_1,\alpha_2,\alpha_3)$ and define the differential operator $D^{\alpha}u=\partial_1^{\alpha_1}\partial_2^{\alpha_2}\partial_3^{\alpha_3}u$. For notational simplicity, we adopt the standard conventions
$L^2(\mathcal{O})\equiv L^2$
and $H^m(\mathcal{O})\equiv H^m$.

Furthermore, for \(1<p<\infty\), we define $L^p(\tau,T;X)$ as the space of Bochner-integrable functions mapping from the interval $[\tau, T]$ to the Banach space $X$. The norm is given by
\begin{align}
	\|u\|_{L^p(\tau,T;X)}=\left(\int_{\tau}^{T}\|u\|_{X}^p dt\right)^{\frac{1}{p}}, \notag
\end{align}
and for convenience we denote $\|u\|_{L^p(\tau,T;X)}$ by $\|u\|_{L^pX}$.
Let $P_{\sigma}:L^2(\mathcal{O})\to H$ be the Leray-Helmholtz orthogonal projection. Let $A:\mathcal{D}(A)\subset H\to H$ be the Stokes operator defined on the domain $\mathcal{D}(A):=H^2(\mathcal{O})\cap V$, where
\begin{align}\label{A def}
	A:=-P_{\sigma}\Delta.
\end{align}
The operator \(A\) is positive-definite, self-adjoint. By the Hilbert-Schmidt theorem, there exists a complete orthonormal system$\left\{w_j\right\}_{j=1}^{\infty}$ in $H$ consisting of eigenfunctions of the compact operator $A^{-1}$.
These eigenfunctions simultaneously serve as eigenfunctions of \(A\) itself.
The corresponding eigenvalues $\left\{\lambda_j\right\}_{j=1}^{\infty}$ satisfy $0<\lambda_1\leq\lambda_2\leq...\to\infty$ (as \(j\to\infty\)) with
$Aw_j=\lambda_jw_j$ holding for each \(j\).
Furthermore, the divergence-free condition allows us to define the \(V\)-norm via
\begin{align}
	\langle Au,u\rangle=\|A^{\frac{1}{2}}u\|_{L^2}=\|\nabla u\|_{L^2}^2.\notag
\end{align}
Let us recall the Poincar\'{e} inequalities:
\be
	\|u\|_{L^2}^2\leq\lambda_1^{-1}\|\nabla u\|_{L^2}^2 \quad \text{for all}\ u\in V, \notag\\
	\|\nabla u\|_{L^2}^2\leq\lambda_1^{-1}\|Au\|_{L^2}^2 \quad \text{for all}\ u\in\mathcal{D}(A).
\ee
We also recall the Gagliardo-Nirenberg-Sobolev interpolation inequality in $\mathbb{R}^n$, that is, for $1\leq p,q< \infty$, $\frac{1}{p}=\frac{\theta}{q}+(1-\theta)(\frac{1}{2}-\frac{|\alpha|}{n})$, $\theta\in[0,1]$, there exists a constant \(C>0\) such that
\begin{align}
	\|u\|_{L^p}\leq C\|u\|_{L^q}^\theta \|D^\alpha u\|_{L^2}^{1-\theta}.\notag
\end{align}
Let $P_m$ denote the orthogonal projection onto the first $m$ eigenfunctions of the Stokes operator $A$, i.e. $P_m u=\sum_{j=1}^{m}u_jw_j$, where $u_j := (u, w_j)$ is the Fourier coefficients of $u$. This projection satisfies the fundamental estimate:
\begin{align}\label{I-Pm}
	\|(I-P_m)u\|_{L^2}^2\leq\lambda_m^{-s}\|u\|_{H^s}^2,
\end{align}
for all $u\in H^s(\mathcal{O}),s>0$.

The nonlinear term $B(u,v)$ is defined by
\begin{align}
	B(u,v):=P_\sigma((\nabla\times v)\times u),\notag
\end{align}
for $u\in C_c^\infty(\mathcal{O})$ and $v\in\mathcal{V}$.
The operator admits a continuous extension to a bounded bilinear operator $B:H_0^1\times V\to V'$. Analogously, define the associated trilinear operator $b:H_0^1\times V\times V\to\mathbb{R}$ by
\begin{align}
	b(u,v,w):=\langle B(u,v),w\rangle,\notag
\end{align}
for all $u\in H_0^1$ and $v,w\in V$.

To recover the pressure term in systems \eqref{original}, we use de Rham's results (see  \cite{temam2001navier},\cite{wang1993remark}), which show that for $f\in C_c^\infty(\mathcal{O})$,
\begin{align}\label{f=nablap}
	f=\nabla p\ \text{for some}\ p\in C_c^\infty(\mathcal{O})\quad \text{if and only if}\quad \langle f,v\rangle=0\ \text{for all}\ v\in\mathcal{V}.
\end{align}
In order to prove the existence of the pullback attractor, we make the assumptions on the external force $f\in L_{\loc}^2(\mathbb{R},H)$:
\begin{enumerate}
	\item[(A1)] There exists a fixed constant  $\alpha\in(0,\nu\lambda_1)$ such that
$$\int_{-\infty}^{\tau}e^{\alpha s}\|f(s, \cdot)\|_{L^2}^2ds<\infty,\quad\text{for all}\ \tau\in\mathbb{R};$$
	\item[(A2)] For every $c>0$, $\lim\limits_{t \to -\infty}e^{ct}\int_{-\infty}^{0}e^{\alpha s}\|f(s+t, \cdot)\|_{L^2}^2ds=0.$
\end{enumerate}	
Furthermore,	
\begin{enumerate}
	\item[(A3)] For fixed $\epsilon>0$, \(\|\zeta^\epsilon\|_{L^\infty}<\nu\sqrt{\frac{\lambda_1}{2}}\).
\end{enumerate}

To handle the calming function in the nonlinear term, we first establish its fundamental properties in the following lemma, which are essential for the subsequent analysis of the bilinear form and energy estimates. The following lemma is adapted from \cite{enlow2024calmedNS}.
\begin{lem}\label{lem:int b is linear}
	Let $\zeta^\epsilon$ be the calming function that satisfies condition 1 of Definition~\ref{def:calming function}.
\begin{enumerate}
	\item[(1)] For any $u\in L^p(\mathcal{O})$ with $p \in[1,\infty]$, we have $\zeta^\epsilon(u)\in L^p(\mathcal{O})$, and $\zeta^\epsilon$ is Lipschitz continuous in $L^p$ with Lipschitz constant $C_\zeta$.
	\item[(2)] Define a mapping $I:L^2(\tau,T;H)\times L^2(\tau,T;V)\times L^2(\tau,T;H)\to\mathbb{R}$ by
	\begin{align}
		I(u,v,w)=\int_{\tau}^{T}b(\zeta^\epsilon(u),v,w)dt,\notag
	\end{align}
	which is linear and continuous in $v$ and $w$.
\end{enumerate}
\end{lem}

\subsection{Pullback attractors}
We recall the fundamental concepts on the theory of pullback attractors from \cite{kloeden2007flattening} and \cite{wang2012sufficient}. Let $(\Omega,\mathcal{F},\mathbb{P})$ be a probability space, where $\mathbb{P}$ is the Wiener measure on $(\Omega,\mathcal{F})$, and $(X,d)$ be a Polish space with its Borel $\sigma$-algebra $\mathcal{B}(X)$. The Wiener shift $\left\{\theta_t\right\}_{t\in\mathbb{R}}$ on the probability space $(\Omega,\mathcal{F},\mathbb{P})$ defines a metric dynamical system through the transformation:
$$\theta_t\omega(\cdot)=\omega(t+\cdot)-\omega(t),t\in\mathbb{R}.$$
This makes  $(\Omega,\mathcal{F},\mathbb{P},\left\{\theta_t\right\}_{t\in\mathbb{R}})$ a measure-preserving dynamical system.
\begin{defn}\label{def:continuous cocycle}
	Let $(\Omega,\mathcal{F},\mathbb{P},\left\{\theta_t\right\}_{t\in\mathbb{R}})$ be a metric dynamical system. A mapping  $\Phi:\mathbb{R}^+\times\mathbb{R}\times\Omega\times X\to X$ is called a continuous cocycle on $X$ over $\mathbb{R}$ and $(\Omega,\mathcal{F},\mathbb{P},\left\{\theta_t\right\}_{t\in\mathbb{R}})$ if for all $t,s\in\mathbb{R}^+$, $\tau\in\mathbb{R}$ and $\omega\in\Omega$, the following conditions hold:
	\begin{description}
		\item[(i)]$\Phi(\cdot,\tau,\cdot,\cdot):\mathbb{R}^+\times\Omega\times X\to X$ is $(\mathcal{B}(\mathbb{R}^+)\times\mathcal{F}\times \mathcal{B}(X), \mathcal{B}(X))$–measurable;
		\item[(ii)]$\Phi(0,\tau,\omega,\cdot)=\operatorname{id}$ on X;
		\item[(iii)]$\Phi(t+s,\tau,\omega,\cdot)= \Phi(t,\tau+s,\theta_s \omega,\cdot) \circ \Phi(s,\tau,\omega,\cdot)$;
		\item[(iv)]$\Phi(t,\tau,\omega,\cdot):X\to X$ is continuous.
	\end{description}
\end{defn}
    Let $\mathcal{D}$ be a collection of some families of nonempty subsets of $X$ parameterized by $\tau\in\mathbb{R}$ and $\omega\in\Omega$, that is, $$\mathcal{D}=\left\{D=\left\{D(\tau,\omega)\subseteq X:D(\tau,\omega)\neq\emptyset,\tau\in\mathbb{R},\omega\in\Omega\right\}\right\}.$$

\begin{defn}\label{def:inclusion-closed}
	A collection $\mathcal{D}$ of some families of nonempty subsets of $X$ is said to be inclusion-closed, if for all  $\tau\in\mathbb{R}$, $\omega\in\Omega$ and  $D=\left\{D(\tau,\omega):\tau\in\mathbb{R}
	,\omega\in\Omega \right\}\in\mathcal{D}$, the family
	\begin{align}
		\left\{B(\tau,\omega):B(\tau,\omega)\ \text{is a nonempty subset of}\ D(\tau,\omega)  \right\}\in \mathcal{D}.\notag
	\end{align}
\end{defn}

\begin{defn}\label{def:tempered}
	A random bounded set $D=\{D(\tau,\omega):\tau\in\mathbb{R},\omega\in \Omega\}$ of $X$ is called tempered, if for $\mathbb{P}$-a.e., $\omega\in\Omega$, $\tau\in \mathbb{R}$,
	$$\lim\limits_{t\to-\infty}e^{\gamma t}|D(\tau,\theta_{-t}\omega)|_X=0,\\\ for\ all\ \gamma>0,$$
where $|D(\tau,\theta_{-t}\omega)|_X=\sup_{x\in D(\tau,\theta_{-t}\omega) }\|x\|_X.$
\end{defn}
\begin{defn}
	Let $\mathcal{D}$ be a collection of some families of nonempty subsets of $X$ and $$K=\left\{K(\tau,\omega):\tau \in \mathbb{R},\omega\in\Omega \right\}\in \mathcal{D}.$$ Then, $K$ is called a $\mathcal{D}$-pullback absorbing set for $\Phi$ if for all $\tau\in \mathbb{R},\omega\in\Omega$ and for every $B\in \mathcal{D}$, there exists $T=T(B,\tau,\omega)>0$ such that $$\Phi(t,\tau-t,\theta_{-t}\omega,B(\tau-t,\theta_{-t}\omega))\subseteq K(\tau,\omega)\ \ for\ all\ t\geq T.$$
\end{defn}

\begin{defn}
	Let $\mathcal{D}$ be a collection of some families of nonempty subsets of $X$. An RDS \(\Phi\) on \(X\) is said to be \(\mathcal{D}\)-pullback flattening if for \(\mathbb{P}\)-a.e. \(\omega\in\Omega\), \(\delta>0\), \(D\in\mathcal{D}\), there exists \(T=T(\tau,\omega,D,\delta)\) and a finite-dimensional space \(X_\delta \subseteq X\), such that for a bounded projection \(P_\delta : X\to X_\delta\), the following conditions hold:
	\item[(i)] \(P_\delta \biggl(\bigcup\limits_{t\geq T}\Phi\left(t,\tau-t,\theta_{-t}\omega,D(\tau-t,\theta_{-t}\omega)\right)\biggr)\) is bounded in \(X\),\\
	\item[(ii)] \(\left\|\left(I-P_\delta\right)\biggl(\bigcup\limits_{t\geq T}\Phi\left(t,\tau-t,\theta_{-t}\omega,D(\tau-t,\theta_{-t}\omega)\right)\biggr)\right\|_{X}<\delta\).
\end{defn}

\begin{defn}\label{def:pullback attractor}
	Let $\mathcal{D}$ be a collection of some families of nonempty subsets of $X$. And $$\mathcal{A}=\left\{\mathcal{A}(\tau,\omega):\tau\in\mathbb{R},\omega\in\Omega\right\}\in\mathcal{D}.$$ Then, $\mathcal{A}$ is called a $\mathcal{D}$-pullback attractor for $\Phi$ if the following conditions are satisfied:
	\begin{description}
		\item[(i)] $\mathcal{A}$ is measurable and $\mathcal{A}(\tau,\omega)$ is compact for all $\tau \in \mathbb{R}$ and $\omega\in \Omega$;
		\item[(ii)] $\mathcal{A}$ is invariant, i.e. for every $\tau\in\mathbb{R}$ and $\omega\in\Omega$,
		\be
			\Phi(t,\tau,\omega,\mathcal{A(\tau,\omega)})=\mathcal{A}(\tau + t,\theta_t \omega)\quad \text{for all}\ t\geq0;
		\ee
		\item[(iii)] $\mathcal{A}$ attracts every member of $\mathcal{D}$, i.e. for every $B=\left\{B(\tau,\omega):\tau\in\mathbb{R},\omega\in\Omega\right\}\in\mathcal{D}$ and for every $\tau\in\mathbb{R}$ and $\omega\in\Omega$,
		$$\lim\limits_{t\to\infty}\dist\left(\Phi(t,\tau-t,\theta_{-t}\omega,B(\tau-t,\theta_{-t}\omega)),\mathcal{A}(\tau,\omega)\right)=0.$$
	\end{description}
\end{defn}
Let \(X\) be a uniformly convex Banach space. We use the following result to establish the existence of a $\mathcal{D}$-pullback attractor (see e.g. \cite{kloeden2007flattening}).
\begin{prop}\label{prop:pullback attractor}
	Let \(X\) be a uniformly convex Banach space, $\mathcal{D}$ be a collection of some families of nonempty subsets of $X$ and $\Phi$ be a continuous RDS on $X$ over $(\Omega,\mathcal{F},\mathbb{P},\left\{\theta_t\right\}_{t\in\mathbb{R}})$. If $B$ is a closed measurable $\mathcal{D}$-pullback absorbing set for $\Phi$ in $\mathcal{D}$ and $\Phi$ is  $\mathcal{D}$-pullback flattening in $X$, then $\Phi$ has a unique $\mathcal{D}$-pullback attractor $\mathcal{A}$.
\end{prop}

\section{Global well-posedness of solutions for the calmed system}\label{sect3}
This section establishes the global well-posedness of solutions for the calmed 3D stochastic Navier-Stokes equation with additive noise in \(V\).
To convert the stochastic differential equation into a random one,
we need the Ornstein-Uhlenbeck (O-U) process. Let
$$z(\theta_t \omega)=-\gamma\int_{-\infty}^{0}e^{\gamma s} (\theta_t \omega)(s)ds, \quad t \in \mathbb{R},$$
be the O-U process, which is the solution to stochastic differential equation
$$dz+\gamma zdt=dW(t).$$
Then, there exists a $\theta_t$-invariant set (still denoted as) $\Omega$ of full measure such that
\begin{enumerate}
	\item[(i)] $z(\theta_t \omega)$ is continuous with respect to $t$,
	\item[(ii)] For all \(\omega\in\Omega\),
	\be\label{z polynomial growth}
		\lim\limits_{t\to \pm\infty}\frac{|z(\theta_t \omega)|}{t}=0, \ \ \
		\lim\limits_{t\to \pm\infty}\frac{1}{t}\int_{0}^{t}z(\theta_r \omega)dr=0.
	\ee
\end{enumerate}
Let
\begin{align}\label{transformation}
	v(t)=u(t)-h(x)z(\theta_t \omega), \omega \in \Omega.
\end{align}
Then equation \eqref{calmed u} can be written as:
\be\label{calmed OU}
	\frac{dv}{dt}+\left(\nabla \times (v+hz(\theta_t\omega))\right)\times \zeta^{\epsilon}(v+hz(\theta_t\omega))+\nabla \pi
	=\nu \triangle \left(v+hz(\theta_t\omega)\right)+f+\gamma h(x)z(\theta _t \omega).
\ee
\begin{defn}\label{solution}
	Assume that $f\in L^2(\tau,T;H)$, $v_0\in V$, $\tau\in\mathbb{R}$ and $T>\tau$ be any fixed time. We call $v$ a solution to \eqref{calmed OU} on the time interval $[\tau,T]$, if for $\mathbb{P}$-a.s., $\omega\in\Omega$,
	\begin{align}
		v\in  L^2(\tau,T;H^2\cap V)\cap C([\tau,T];V)\cap H^1(\tau,T;H),\notag
	\end{align}
	and it satisfies
	\be[weak formulation]
\left( \frac{d}{dt} v,w \right)+&\left( \nu A\left(v+hz(\theta_t\omega)\right),w \right) \notag\\
		=&\left\langle -B\left(\zeta^\epsilon(v+hz(\theta_t\omega)),v+hz(\theta_t\omega)\right),w \right\rangle +\left( f+\gamma hz(\theta_t\omega),w \right)
	\ee	for all $w \in V$.
\end{defn}

\begin{thm}\label{thm:existence v}
	Let $u_0 \in V$, $\omega\in\Omega$, $\tau\in\mathbb{R}$, $T>\tau$, and $f \in L^2(\tau,T;H)$ be given. Then, a unique solution to \eqref{calmed u} exists on $[\tau,T]$ in $V$. Moreover, this solution depends continuously on its initial data.
\end{thm}
\begin{proof}
Formally, if $u$ solves \eqref{calmed u}, then $v(t)=u(t)-h(x)z(\theta_t\omega)$ satisfies \eqref{calmed OU} for $t\in[\tau,T]$ with $T>\tau$.
We begin by establishing the existence of solutions to \eqref{calmed OU} through Galerkin approximation.
Let \(\{w_j\}_{j=1}^{\infty}\) be the eigenfunctions of Stokes operator \(A\). Define the finite-dimensional subspace \(V_m={\rm span}\{w_1,w_2,...,w_m\}\subset V\), and consider \(v_m=\sum_{j=1}^{m}(v,w_j) w_j\) and \(h_m=\sum_{j=1}^{m}h_j w_j\) in \(V_m\).
We study the finite-dimensional system with approximate solutions $v_m$ on some maximal existence interval $[\tau,T_m]$ with $T_m>\tau$,
\be\label{galerkin}
\begin{cases}
	\frac{d}{dt}v_m = -\nu A(v_m +h_mz(\theta_t\omega))-P_m B\left(\zeta^\epsilon (v_m+h_mz(\theta_t\omega)),v_m+h_mz(\theta_t\omega)\right)\\
	\qquad\qquad+P_mf+\gamma h_mz(\theta_t\omega)\\
	v_m(x,\tau)=P_m v_0 (x).
\end{cases}
\ee
Since $\zeta^\epsilon$ is Lipschitz, the system \eqref{galerkin} is locally Lipschitz in $V_m$ for $v_0 \in V$. For each $m \in \mathbb{N}$, a unique solution to \eqref{galerkin} exists for some $T_m >\tau$.
We take the inner product of \eqref{galerkin} with $Av_m$ and use H\"{o}lder's inequality and Young's inequality to obtain
\begin{align}
	&\frac{1}{2}\frac{d}{dt}\|\nabla v_m\|_{L^2}^2+\nu	\|Av_m\|_{L^2}^2 \notag\\
	=\ &b\left(\zeta^\epsilon(v_m+h_mz(\theta_t\omega)),v_m+h_mz(\theta_t\omega),Av_m\right)+\left( f+\gamma h_mz(\theta_t\omega),Av_m\right)-\nu\left(z(\theta_t\omega)Ah_m,Av_m\right)\notag\\
	\leq\ &\frac{\nu}{2}\|Av_m\|_{L^2}^2+C_\nu\|\zeta^\epsilon\|_{L^\infty}^2\|\nabla v_m\|_{L^2}^2+C_\nu\left(\|\zeta^\epsilon\|_{L^\infty}^2|z(\theta_t\omega)|^2\|\nabla h_m\|_{L^2}^2+\|f\|_{L^2}^2\right.\notag\\
	&\left.+|z(\theta_t\omega)|^2\|h_m\|_{L^2}^2+\nu^2|z(\theta_t\omega)|^2\|\Delta h_m\|_{L^2}^2\vphantom{\|\zeta^\epsilon\|_{L^\infty}^2}\right)\notag\\
	\leq\ &\frac{\nu}{2}\|Av_m\|_{L^2}^2+C_\nu\|\zeta^\epsilon\|_{L^\infty}^2\|\nabla v_m\|_{L^2}^2+C_\nu\|f\|_{L^2}^2+C_{\nu,\epsilon}|z(\theta_t\omega)|^2. \notag
\end{align}
Rearranging these terms yields
\be [after inequality]
\frac{d}{dt}\|\nabla v_m\|_{L^2}^2+\nu\|Av_m\|_{L^2}^2 \leq C_\nu\|\zeta^\epsilon\|_{L^\infty}^2\|\nabla v_m\|_{L^2}^2+C_{\nu,\epsilon}|z(\theta_t\omega)|^2+C_\nu\|f\|_{L^2}^2.
\ee
After omitting the term $\nu \|A v_m\|_{L^2}^2$, we apply Gr\"{o}nwall's inequality to obtain
\begin{align} \label{after gronwall}
\|\nabla v_m\|_{L^2}^2
\leq & \|\nabla v_0\|_{L^2}^2 e^{C_\nu\|\zeta^\epsilon\|_{L^\infty}^2T_m} + \int_{\tau}^{t}e^{-C_\nu\|\zeta^\epsilon\|_{L^\infty}^2(s-t)}\left[C_\nu\|f(s)\|_{L^2}^2+C_{\nu,\epsilon}\left|z(\theta_s\omega)\right|^2\right]ds,
\end{align}
for a.e. $t\in[\tau,T_m]$.
In fact, for any given \(T>\tau\), \eqref{after gronwall} holds for all $m \in \mathbb{N}$ when $T_m=T$ by a standard bootstrapping argument \cite{Tao2006NonlinearDE}. Furthermore,  $v_m$ is uniformly bounded in $L^{\infty}(\tau,T;V)$ since $z(\theta_t \omega)$ is continuous with respect to $t$ and \(f\in L^2(\tau,T;H)\). We integrate \eqref{after inequality} with respect to time $t$ over the interval $[\tau,T]$,
\be [4.3]
\nu\int_{\tau}^{T}\|Av_m\|_{L^2}^2dt
\leq\ &\|\nabla v_0\|_{L^2}^2+C_{\nu,\epsilon}\int_{\tau}^{T}(\|\nabla v_m\|_{L^2}^2+\left|z(\theta_t\omega)\right|^2)dt	+C_\nu\|f\|_{L^2L^2}^2.
\ee
Therefore, from \eqref{after gronwall} and \eqref{4.3}, we deduce that $v_m$ is bounded in $L^2(\tau,T;H^2\cap V)$ independently of $m$.

Next we show that $\frac{d}{dt}v_m\in L^2(\tau,T;H)$. Integrating \eqref{galerkin} in time $[\tau,T]$ yields
\be
&\int_{\tau}^{T}\left\|\frac{d}{dt}v_m\right\|_{L^2}^2dt\\		
=\
&\int_{\tau}^{T}\|-\nu A\left(v_m+h_mz(\theta_t\omega)\right)-P_mB(\zeta^\epsilon(v_m+h_mz(\theta_t\omega)),v_m+h_mz(\theta_t\omega))\\
 &\qquad+P_mf+\gamma h_mz(\theta_t\omega)\|_{L^2}^2dt\notag\\
\leq\
&C_{\nu,\epsilon}\int_{\tau}^{T}\left( \|Av_m\|_{L^2}^2+|z(\theta_t\omega)|^2+\|\nabla v_m\|_{L^2}^2 +\|f\|_{L^2}^2\right)dt\notag\\
<\ &\infty.
\ee
Therefore, $\frac{d}{dt}v_m\in L^2(\tau,T;H)$ holds.

Combining above bounds and applying the Alaoglu weak-star compactness theorem, we conclude that there exists $v \in L^{\infty}(\tau,T;V) \cap L^2(\tau,T;H^2\cap V)$ and a subsequence which still labeled as $v_m$ such that
\begin{align}
	v_m &\to v \ \text{weak star in}\ L^{\infty}(\tau,T;V), \label{3.6}\\
	v_m &\to v \ \text{weakly in}\ L^2(\tau,T;H^2\cap V), \label{3.7}\\
	\frac{d}{dt} v_m &\to \frac{d}{dt}v \ \text{weakly in}\ L^{2}(\tau,T;H), \label{3.8}
\end{align}
Using the Aubin-Lions lemma and the embedding $H^2\cap V \hookrightarrow V \hookrightarrow H$, we obtain a subsequence (still labeled as) $v_m$ such that
\begin{align}
	v_m \to v \ \text{strongly in}\ L^2(\tau,T;V).\notag
\end{align}

Next we prove that $v$ satisfies \eqref{weak formulation}. Let $w \in V$, set $\alpha_m=v-v_m$, our goal is to show that
\be[a1]
	&\left( \frac{d}{dt}v ,w\right) +  b\left(\zeta^\epsilon\left(v+hz(\theta_t\omega)\right),v+hz(\theta_t\omega),w\right)  +\left(\nu\nabla \left(v+hz(\theta_t\omega)\right),\nabla w \right)\\
	&+\left (f+\gamma hz(\theta_t\omega),w \right) -
	\left( \frac{d}{dt}v_m,w \right)- b\left(\zeta^\epsilon\left(v_m+h_mz(\theta_t\omega)\right),v_m+h_mz(\theta_t\omega),P_mw \right)\\
	&-\left( \nu\nabla (v_m+h_mz(\theta_t\omega)),\nabla w \right)-\left( P_mf+\gamma h_mz(\theta_t\omega),w \right)
\ee
tends to 0 as $m \to \infty$. We rewrite \eqref{a1} as follows:
\be
	&\left( \frac{d}{dt}\alpha_m,w \right)+\nu ( \nabla( \alpha_m+hz(\theta_t\omega)-h_mz(\theta_t\omega)), \nabla w)\\ &+b(\zeta^\epsilon(v+hz(\theta_t\omega)),\alpha_m+hz(\theta_t\omega)-h_mz(\theta_t\omega),w)\\
	&+b(\zeta^\epsilon(v+hz(\theta_t\omega)),v_m+h_mz(\theta_t\omega),w)
-b(\zeta^\epsilon(v_m+h_mz(\theta_t\omega)),v_m+h_mz(\theta_t\omega),w) \\
&+b(\zeta^\epsilon(v_m+h_mz(\theta_t\omega)),v_m
+h_mz(\theta_t\omega),(I-P_m)w)+\left( (I-P_m)f+\gamma (h-h_m)z(\theta_t\omega),w\right).
\ee
From \eqref{3.7}, \eqref{3.8} and Lemma~\ref{lem:int b is linear}, we obtain
\begin{align}
	\lim\limits_{m\to \infty}\int_{\tau}^{T} \left( \frac{d}{dt} \alpha_m,w \right)+\nu( \nabla (\alpha_m+hz(\theta_t\omega)-h_mz(\theta_t\omega)),\nabla w ) dt=0, \notag
\end{align}
and
\begin{align}
	\lim\limits_{m \to \infty} \int_{\tau}^{T}b(\zeta^\epsilon(v+hz(\theta_t\omega)),\alpha_m+hz(\theta_t\omega)-h_mz(\theta_t\omega),w)dt=0. \notag
\end{align}
Now using the Gagliardo-Nirenberg-Sobolev inequality, H\"{o}lder's inequality and the properties of $\zeta^\epsilon$, we obtain
\be
&\int_{\tau}^{T}\left(b(\zeta^\epsilon(v+hz(\theta_t\omega)),
v_m+h_mz(\theta_t\omega),w)-b(\zeta^\epsilon(v_m+h_mz(\theta_t\omega)),v_m+h_mz(\theta_t\omega),w)\right)dt \\
\leq\
&\int_{\tau}^{T} \|\alpha_m+hz(\theta_t\omega)-h_mz(\theta_t\omega)\|_{L^3}\|\nabla \left(v_m+h_mz(\theta_t\omega)\right)\|_{L^2}\|w\|_{L^6}dt\\
	\leq\
	&C\int_{\tau}^{T}\|\alpha_m+hz(\theta_t\omega)-h_mz(\theta_t\omega)\|_{L^2}^{\frac{1}{2}}\|\nabla (\alpha_m+hz(\theta_t\omega)-h_mz(\theta_t\omega))\|_{L^2}^{\frac{1}{2}}\\
	&\cdot\|\nabla (v_m+h_mz(\theta_t\omega))\|_{L^2}\|\nabla w\|_{L^2}dt\\
	\leq\
	&C\|\nabla w\|_{L^2}\|\alpha_m+hz(\theta_t\omega)-h_mz(\theta_t\omega)\|_{L^2 L^2}^{\frac{1}{2}}\|\nabla (\alpha_m+hz(\theta_t\omega)-h_mz(\theta_t\omega))\|_{L^2 L^2}^{\frac{1}{2}}\\
	&\cdot\|\nabla \left(v_m+h_mz(\theta_t\omega)\right)\|_{L^2 L^2}.
\ee
Since $\|\nabla (v_m+h_mz(\theta_t\omega))\|_{L^2 L^2}$ is bounded, $v_m\to v$ strongly in $L^2(\tau,T;V)$ and \(L^2(\tau,T;V)\) is continuously embedded in \(L^2(\tau,T;H)\), we have
\begin{align}
	\lim\limits_{m \to \infty}\int_{\tau}^{T}&\left(b\left(\zeta^\epsilon\left(v+hz(\theta_t\omega)\right),v_m+h_mz(\theta_t\omega),w\right)\right.\notag\\
	&\left.-b\left(\zeta^\epsilon\left(v_m+h_mz(\theta_t\omega)\right),v_m+h_mz(\theta_t\omega),w\right)\right)dt=0.\notag
\end{align}
By \eqref{I-Pm}, we get
\begin{align}
	&\lim\limits_{m\to \infty}\left|\int_{\tau}^{T}b\left(\zeta^\epsilon(v_m+h_mz(\theta_t\omega)),v_m+h_mz(\theta_t\omega),(I-P_m)w\right)dt\right|\notag\\
	\leq &\lim\limits_{m\to \infty}\|\zeta^\epsilon\|_{L^\infty}\lambda_m^{-\frac{1}{2}}\|w\|_{H^1} \int_{\tau}^{T}\|\nabla(v_m+h_mz(\theta_t\omega))\|_{L^2}dt \notag\\
	\leq &\lim\limits_{m\to \infty}\|\zeta^\epsilon\|_{L^\infty}\lambda_m^{-\frac{1}{2}}\|w\|_{H^1} T\sup_{m \in \mathbb{N}}\|v_m+h_mz(\theta_t\omega)\|_{L^2 V} \notag\\
	= &\ 0, \notag
\end{align}
and
\begin{align}
	\lim\limits_{m \to \infty}\int_{\tau}^{T} \left( (I-P_m)f+\gamma (h-h_m)z(\theta_t\omega),w\right) dt = 0, \notag
\end{align}
hence passing the limit in \eqref{galerkin} yields \eqref{weak formulation}. From these estimates above we conclude that a subsequence of solutions $v_m$ to \eqref{galerkin} converges to a solution $v$ to \eqref{calmed OU}. We must verify that $v$ is time-continuous and satisfies the initial condition. This follows directly from the Aubin-Lions Compactness Theorem(cf.\cite{robinson2001infinite}) that $v\in C([\tau,T];V)$.
\par
Now, we show that $v$ satisfies the initial data. For all $w \in V$, we have
\be [3.13]
	\left( \frac{d}{dt} v,w\right)=\frac{d}{dt}(v,w)
	=&-\left\langle B\left(\zeta^\epsilon(v+hz(\theta_t\omega)),v+hz(\theta_t\omega)\right),w\right\rangle\notag\\
	&+\nu\left(\Delta\left(v+hz(\theta_t\omega)\right),w\right)+\left(f+\gamma hz(\theta_t\omega),w\right).
\ee
Let $\psi \in C^1\left([\tau,T]\right)$ be a test function such that $\psi(\tau)=1$ and $\psi(T)=0$. Multiplying \eqref{3.13} by $\psi$ and integrating over $[\tau,T]$, we then apply integration by parts to obtain
\be[3.14]	-\int_{\tau}^{T}(v,w)\psi'(t)dt=&-\int_{\tau}^{T}\langle B(\zeta^\epsilon(v+hz(\theta_t\omega)),v+hz(\theta_t\omega)),w\rangle \psi(t)dt  \notag\\&+\nu\int_{\tau}^{T}\left(\Delta\left(v+hz(\theta_t\omega)\right),w\right) \psi(t)dt\notag\\&+\int_{\tau}^{T}(f+\gamma hz(\theta_t\omega),w)\psi(t)dt +(v(\tau),w).
\ee
On the other hand, by taking the inner product of \eqref{galerkin} with $w$ and integrating in time with $\psi$, we obtain
\begin{align}
	-\int_{\tau}^{T}(v_m,w)\psi'(t)dt
	=&-\int_{\tau}^{T}\left\langle P_mB\left(\zeta^\epsilon\left(v_m+h_mz(\theta_t\omega)\right),v_m+h_mz(\theta_t\omega)\right),w\right\rangle\psi(t)dt\notag\\
	&-\nu\int_{\tau}^{T}\left(A(v_m+h_mz(\theta_t\omega)),w\right)\psi(t)dt  \notag\\
	&+\int_{\tau}^{T}\left(P_mf+\gamma h_mz(\theta_t\omega),w\right)\psi(t)dt+(v_{0}^m(x),w). \notag
\end{align}
Taking the limit as $m \to \infty$ subsequently results in
\be[3.15]
	-\int_{\tau}^{T}(v,w)\psi'(t)dt=&-\int_{\tau}^{T}\left\langle B\left(\zeta^\epsilon(v+hz(\theta_t\omega)),v+hz(\theta_t\omega)\right),w\right\rangle\psi(t)dt \notag\\&-\nu\int_{\tau}^{T}\left(A(v+hz(\theta_t\omega)),w\right)\psi(t)dt \notag\\&+\int_{\tau}^{T}(f+\gamma hz(\theta_t\omega),w)\psi(t)dt+(v_0(x),w).
\ee
Comparing \eqref{3.14} with \eqref{3.15}, we obtain
$$(v(\tau)-v_0(x),w)=0,\ \text{for all}\ w\in V,$$
which implies that $v$ satisfies $v(\tau)=v_0(x)$.
Thus, $v\in  L^2(\tau,T;H^2\cap V)\cap C([\tau,T];V)\cap H^1(\tau,T;H)$ with
$v_0\in V$ is a solution to \eqref{calmed OU}.
Let $u(t)=v(t)+h(x)z(\theta_t\omega)$ for $t\in[\tau,T]$ for any fixed $T>\tau$. Then we find that $u$ is continuous in time. In addition, it follows from \eqref{calmed OU} that $u$ is a solution to \eqref{calmed u}.\par
Now we prove that the solution \(u\) is unique.
Set $\tilde{u}=u_1-u_2$, $\tilde{u_0}=u_0^1-u_0^2$, where $u_1$ and $u_2$ are solutions of \eqref{calmed u} on \([\tau,T]\) with  identical initial data $u_0^1=u_0^2$. After taking the difference equation with $\tilde{u}$, we have
\begin{align}\label{3.16}
	\frac{1}{2}\frac{d}{dt}\|\tilde{u}\|_{L^2}^2+\nu\|\nabla \tilde{u}\|_{L^2}^2
	=b(\zeta^\epsilon(u_2)-\zeta^\epsilon(u_1),u_2,\tilde{u})-b(\zeta^\epsilon(u_1),\tilde{u},\tilde{u}).
\end{align}
Then we use condition 1 of Definition~\ref{def:calming function}, H\"{o}lder's inequality, the Gagliardo-Nirenberg-Sobolev inequality and Poincar\'e's inequality,
\begin{align}
	|b(\zeta^\epsilon(u_2)-\zeta^\epsilon(u_1),u_2,\tilde{u})|=&\|\tilde{u}\|_{L^3}\|\nabla u_2\|_{L^6}\|\tilde{u}\|_{L^2} \notag\\
	\leq &\ C\|\tilde{u}\|_{L^2}^{\frac{3}{2}}\|\nabla \tilde{u}\|_{L^2}^\frac{1}{2}\|\Delta u_2\|_{L^2}\notag\\
	\leq &\ C_\nu\|\Delta u_2\|_{L^2}^\frac{4}{3}\|\tilde{u}\|_{L^2}^2+\frac{\nu}{4}\|\nabla\tilde{u}\|_{L^2}^2.\notag
\end{align}
Also for the second term on the right-hand side of \eqref{3.16}, we have
\begin{align}
	|b(\zeta^\epsilon(u_1),\tilde{u},\tilde{u})|
	\leq C_{\nu,\epsilon}\|\tilde{u}\|_{L^2}^2+\frac{\nu}{4}\|\nabla\tilde{u}\|_{L^2}^2.\notag
\end{align}
Combining the above estimates, we obtain
\begin{align}
	\frac{d}{dt}\|\tilde{u}\|_{L^2}^2+\nu\|\nabla\tilde{u}\|_{L^2}^2
	\leq  C_{\nu,\epsilon}\|\Delta u_2\|_{L^2}^\frac{4}{3}\|\tilde{u}\|_{L^2}^2.\notag
\end{align}
Since $u_2$ is a solution and $u_2\in L^2(\tau,T;H^2\cap V)$, we obtain
\begin{align}
	C_{\nu,\epsilon}\int_{\tau}^{T}\|\Delta u_2\|_{L^2}^\frac{4}{3}dt<\infty,
\end{align}
then we use Gr\"{o}nwall's inequality,
\begin{align}
	\|\tilde{u}\|_{L^2}^2\leq \exp\left( C_{\nu,\epsilon}\int_{\tau}^{T}\|\Delta u_2\|_{L^2}^\frac{4}{3}dt\right)\|\tilde{u_0}\|_{L^2}^2,
\end{align}
which implies that solutions to \eqref{calmed u} are unique and depend continuously on the initial data.
\end{proof}

\section{Pullback attractors}\label{sect4}
From section \ref{sect3}, we know that for every \(\tau\in\mathbb{R}\), \(u_\tau\in V\), and \(\omega\in\Omega\), equation \eqref{calmed u} has a unique solution \(u(\cdot, \tau, \omega, u_\tau)\). Define a continuous cocycle on $V$ over  $(\Omega,\mathcal{F},\mathbb{P},\left\{\theta_t\right\}_{t\in\mathbb{R}})$ as $\Phi:\mathbb{R}^+\times\mathbb{R}\times\Omega\times V\to V$ such that, for every $t\in\mathbb{R}^+, \ \tau\in\mathbb{R},\ \omega\in\Omega$ and $u_\tau\in V,$
$$\Phi(t,\tau,\omega,u_\tau)=u(t+\tau,\tau,\theta_{-\tau}\omega,u_\tau)
=v(t+\tau,\tau,\theta_{-\tau}\omega,v_\tau)+hz(\theta_t\omega),$$ where $v_\tau=u_\tau-hz(\omega)$.

In this section, we denote by \(\mathcal{D}_V=\{D=\{ D(\tau,\omega)\subseteq V:D(\tau,\omega)\neq\emptyset,\tau\in\mathbb{R},\omega\in\Omega\}\}\) the collection of some families of nonempty subsets of $V$ parameterized by $\tau\in\mathbb{R}$ and $\omega\in\Omega$. We first prove uniform estimates on the solutions of problem \eqref{calmed OU}. Next, we prove that the random dynamical system $\Phi$ associated with problem \eqref{calmed u} has a $\mathcal{D}_V$-pullback absorbing set. By proving that the cocycle \(\Phi\) is $\mathcal{D}_V$-pullback flattening, Proposition~\ref{prop:pullback attractor} implies the existence of a $\mathcal{D}_V$-pullback attractor for $\Phi$. For simplicity, we denote by \(M_\epsilon:=\|\zeta^\epsilon\|_{L^\infty}\).

\subsection{Construction of $\mathcal{D}_V$-pullback absorbing set}
\begin{lem}\label{lem:uniform estimates}
	For every \(\tau\in\mathbb{R}\), \(\omega\in\Omega\) and \(D=\{D(\tau,\omega):\tau\in\mathbb{R},\omega\in\Omega\}\in\mathcal{D}_V\), there exists \(T_1=T_1(\tau,\omega,D)>0\) such that, for all \(t\geq T_1\) and \(s\geq \tau-t\), the solution to \eqref{calmed OU} satisfies
	\be [nabla v estimate]
		\|\nabla v(s,\tau-t,\theta_{-\tau}\omega,v_{\tau-t})\|_{L^2}^2
		\leq
		&\ C_{\nu,\epsilon}e^{-\left(\nu\lambda_1-\frac{2M_\epsilon^2}{\nu}\right)\left(s-\tau\right)}\left[1 +\int_{-\infty}^{s-\tau}e^{\left(\nu\lambda_1-\frac{2M_\epsilon^2}{\nu}\right)r}\|f(r+\tau,\cdot)\|_{L^2}^2dr \vphantom{\int_{-\infty}^{s-\tau}}\right. \\
		&\
		\left.+\int_{-\infty}^{s-\tau}e^{\left(\nu\lambda_1-\frac{2M_\epsilon^2}{\nu}\right)r}|z(\theta_r\omega)|^2dr \vphantom{\int_{-\infty}^{s-\tau}}\right],
	\ee
	where \(v_{\tau-t}\in D(\tau-t,\theta_{-t}\omega)\).
\end{lem}
\begin{proof}
	Taking the action of \eqref{weak formulation} with \(Av\) and then using Young's inequality and H\"{o}lder's inequality, we have
	\begin{align}
		\frac{d}{dt}\|\nabla v\|_{L^2}^2 +\nu\|Av\|_{L^2}^2
		\leq
		\frac{2M_\epsilon^2}{\nu}\|\nabla v\|_{L^2}^2 +\frac{2C}{\nu}\|f\|_{L^2}^2 +\frac{CM_\epsilon^2}{\nu}|z(\theta_t\omega)|^2. \notag
	\end{align}
	By the Poincar\'{e} inequalities, we have
	\begin{align}\label{energy inequality}
		\frac{d}{dt}\|\nabla v\|_{L^2}^2 +\left(\nu\lambda_1-\frac{2M_\epsilon^2}{\nu}\right)\|\nabla v\|_{L^2}^2
		\leq
	 	\frac{2C}{\nu}\|f\|_{L^2}^2 +\frac{CM_\epsilon^2}{\nu}|z(\theta_t\omega)|^2.
	\end{align}
	Multiplying \eqref{energy inequality} by \(e^{\left(\nu\lambda _1-\frac{2M_\epsilon^2}{\nu}\right)t}\), we obtain
	\begin{align}
		\frac{d}{dt}\left(e^{\left(\nu\lambda_1-\frac{2M_\epsilon^2}{\nu}\right)t}\|\nabla v\|_{L^2}^2\right)
		\leq
		&e^{\left(\nu\lambda_1-\frac{2M_\epsilon^2}{\nu}\right)t}\left[\frac{2C}{\nu}\|f\|_{L^2}^2 +\frac{CM_\epsilon^2}{\nu}|z(\theta_t\omega)|^2\right].\notag
	\end{align}
	Then integrating the inequality on \([\tau-t,s]\), we obtain
	\begin{align}
		\|\nabla v(s,\tau-t,\omega,v_{\tau-t})\|_{L^2}^2
		\leq
		&\ e^{\left(\nu\lambda_1-\frac{2M_\epsilon^2}{\nu}\right)\left(\tau-s\right)} e^{-\left(\nu\lambda_1-\frac{2M_\epsilon^2}{\nu}\right)t}\|\nabla v_{\tau-t}\|_{L^2}^2 \notag\\
		&
		+\frac{2C}{\nu}e^{-\left(\nu\lambda_1-\frac{2M_\epsilon^2}{\nu}\right)s}\int_{\tau-t}^{s}e^{\left(\nu\lambda_1-\frac{2M_\epsilon^2}{\nu}\right)r}\|f(r,\cdot)\|_{L^2}^2dr \notag\\
		&
		+\frac{CM_\epsilon^2}{\nu}e^{-\left(\nu\lambda_1-\frac{2M_\epsilon^2}{\nu}\right)s}\int_{\tau-t}^{s}e^{\left(\nu\lambda_1-\frac{2M_\epsilon^2}{\nu}\right)r}|z(\theta_{r}\omega)|^2dr. \notag
	\end{align}
	Replacing \(\omega\) with \(\theta_{-\tau}\omega\), we obtain
	\be[estimate1]
		\|\nabla v(s,\tau-t,\theta_{-\tau}\omega,v_{\tau-t})\|_{L^2}^2
		\leq
		&\
		e^{\left(\nu\lambda_1-\frac{2M_\epsilon^2}{\nu}\right)\left(\tau-s\right)} e^{-\left(\nu\lambda_1-\frac{2M_\epsilon^2}{\nu}\right)t}\|\nabla v_{\tau-t}\|_{L^2}^2 \notag\\
		&\ +\frac{2C}{\nu}e^{-\left(\nu\lambda_1-\frac{2M_\epsilon^2}{\nu}\right)s}\int_{\tau-t}^{s}e^{\left(\nu\lambda_1-\frac{2M_\epsilon^2}{\nu}\right)r}\|f(r,\cdot)\|_{L^2}^2dr \notag\\
		&\ +\frac{CM_\epsilon^2}{\nu}e^{-\left(\nu\lambda_1-\frac{2M_\epsilon^2}{\nu}\right)s}\int_{\tau-t}^{s}e^{\left(\nu\lambda_1-\frac{2M_\epsilon^2}{\nu}\right)r}|z(\theta_{r-\tau}\omega)|^2dr \notag\\
		\leq
		&\ e^{\left(\nu\lambda_1-\frac{2M_\epsilon^2}{\nu}\right)\left(\tau-s\right)} e^{-\left(\nu\lambda_1-\frac{2M_\epsilon^2}{\nu}\right)t}\|\nabla v_{\tau-t}\|_{L^2}^2 \notag\\
		&\
		+\frac{2C}{\nu}e^{-\left(\nu\lambda_1-\frac{2M_\epsilon^2}{\nu}\right)s}\int_{-t}^{s-\tau}e^{\left(\nu\lambda_1-\frac{2M_\epsilon^2}{\nu}\right)(r+\tau)}\|f(r+\tau,\cdot)\|_{L^2}^2dr \notag\\
		&\
		+\frac{CM_\epsilon^2}{\nu}e^{-\left(\nu\lambda_1-\frac{2M_\epsilon^2}{\nu}\right)s}\int_{-t}^{s-\tau}e^{\left(\nu\lambda_1-\frac{2M_\epsilon^2}{\nu}\right)(r+\tau)}|z(\theta_r\omega)|^2dr \notag\\
		\leq
		&\
		e^{\left(\nu\lambda_1-\frac{2M_\epsilon^2}{\nu}\right)\left(\tau-s\right)} e^{-\left(\nu\lambda_1-\frac{2M_\epsilon^2}{\nu}\right)t}\|\nabla v_{\tau-t}\|_{L^2}^2 \notag\\
		&\
		+\frac{2C}{\nu}e^{-\left(\nu\lambda_1-\frac{2M_\epsilon^2}{\nu}\right)\left(s-\tau\right)}\int_{-\infty}^{s-\tau}e^{\left(\nu\lambda_1-\frac{2M_\epsilon^2}{\nu}\right)r}\|f(r+\tau,\cdot)\|_{L^2}^2dr \notag\\
		&\
		+\frac{CM_\epsilon^2}{\nu}e^{-\left(\nu\lambda_1-\frac{2M_\epsilon^2}{\nu}\right)\left(s-\tau\right)}\int_{-\infty}^{s-\tau}e^{\left(\nu\lambda_1-\frac{2M_\epsilon^2}{\nu}\right)r}|z(\theta_r\omega)|^2dr.
	\ee
	Then we estimate the terms on the right-hand side of \eqref{estimate1} one by one.
	Since \(v_{\tau-t}\in D(\tau-t,\theta_{-t}\omega)\) and \(D\in\mathcal{D}_V\), we have
	\begin{align}
		e^{-\left(\nu\lambda_1-\frac{2M_\epsilon^2}{\nu}\right)t}\|\nabla v_{\tau-t}\|_{L^2}^2\to 0, \ \ \ \text{as}\ t\to+\infty,\notag
	\end{align}
	which yields that there exists \(T_1=T_1(\tau,\omega,D)\) such that for all \(t\geq T_1\) and \(s\geq \tau-t\),
	\begin{align}\label{estimate2}
		e^{-\left(\nu\lambda_1-\frac{2M_\epsilon^2}{\nu}\right)t}\|\nabla v_{\tau-t}\|_{L^2}^2\leq1.
	\end{align}
	By (A1) and (A3), we have for all \(t\geq T_1\), \(\tau\in\mathbb{R}\), \(s\geq\tau-t\), \(\omega\in\Omega\),
	\begin{align}\label{estimate3}
		\int_{-\infty}^{s-\tau}e^{\left(\nu\lambda_1-\frac{2M_\epsilon^2}{\nu}\right)r}\|f(r+\tau,\cdot)\|_{L^2}^2dr<+\infty,
	\end{align}
	and by \eqref{z polynomial growth}, we have
	\begin{align}\label{estimate4}
		\int_{-\infty}^{s-\tau}e^{\left(\nu\lambda_1-\frac{2M_\epsilon^2}{\nu}\right)r}|z(\theta_r\omega)|^2dr<+\infty.
	\end{align}
	Combining \eqref{estimate1}, \eqref{estimate2}, \eqref{estimate3}, and \eqref{estimate4}, the proof is complete.
\end{proof}
Lemma~\ref{lem:uniform estimates} directly implies the existence of $\mathcal{D}_V$-pullback absorbing sets for the cocycle $\Phi$.
\begin{lem}\label{lem:absorbing set}
	The cocycle $\Phi$ has a closed measurable $\mathcal{D}_V$-pullback absorbing set \(K_V\in\mathcal{D}_V\) in \(V\) given by
	\begin{align}
		K_V(\tau,\omega)=\left\{u\in V:\|u\|_V^2\leq R_V(\tau,\omega)\right\},\notag
	\end{align}
	where \(R_V(\tau,\omega)\) is defined as
	\begin{align}
		R_V(\tau,\omega)=M_1\left[1+|z(\omega)|^2 +\int_{-\infty}^{0}e^{\left(\nu\lambda_1-\frac{2M_\epsilon^2}{\nu}\right)r}\left(\|f(r+\tau,\cdot)\|_{L^2}^2+|z(\theta_r\omega)|^2\right)dr\right],\notag
	\end{align}
	with \(M_1\) being a positive constant independent of $\tau$ and $\omega$.
	\begin{proof}
		From \eqref{transformation}, we have
		\begin{align}
			u(\tau,\tau-t,\theta_{-\tau}\omega,u_{\tau-t})=v(\tau,\tau-t,\theta_{-\tau}\omega,v_{\tau-t})+h(x)z(\omega),\notag
		\end{align}
		where \(v_{\tau-t}=u_{\tau-t}-h(x)z(\theta_{-t}\omega)\).
		Using Lemma~\ref{lem:uniform estimates} with \(s=\tau\), we find that there exists \(T_1=T_1(\tau,\omega,D)>0\) such that, for all \(t\geq T_1\),
		\begin{align}
			&\|\nabla u(\tau,\tau-t,\theta_{-\tau}\omega,u_{\tau-t})\|_{L^2}^2\notag\\
			\leq
			&\ 2\|\nabla v(\tau,\tau-t,\theta_{-\tau}\omega,v_{\tau-t})\|_{L^2}^2 +2\|h\|_{H^2}^2|z(\omega)|^2\notag\\
			\leq
			&\ C_{\nu,\epsilon}\left[ 1+\int_{-\infty}^{0}e^{\left(\nu\lambda_1-\frac{2M_\epsilon^2}{\nu}\right)r}\|f(r+\tau,\cdot)\|_{L^2}^2dr
			+\int_{-\infty}^{0}e^{\left(\nu\lambda_1-\frac{2M_\epsilon^2}{\nu}\right)r}|z(\theta_r\omega)|^2dr \vphantom{\int_{-\infty}^{0}e^{\left(\nu\lambda_1-\frac{2M_\epsilon^2}{\nu}\right)r}}\right]
			+2\|h\|_{H^2}^2|z(\omega)|^2\notag.
		\end{align}
		Thus, there exists a positive constant \(M_1\) (independent of \(\tau\), \(\omega\) and \(D\)) such that, for all \(t \geq T_1\),
		\begin{align}
			&\|\nabla u(\tau,\tau-t,\theta_{-\tau}\omega,u_{\tau-t})\|_{L^2}^2\notag\\
			\leq\
			&M_1\left[1+|z(\omega)|^2 +\int_{-\infty}^{0}e^{\left(\nu\lambda_1-\frac{2M_\epsilon^2}{\nu}\right)r}\left(\|f(r+\tau,\cdot)\|_{L^2}^2+|z(\theta_r\omega)|^2\right)dr\right]\notag\\
			:=
			&\ R_V(\tau,\omega).\notag
		\end{align}
		This implies that for all \(t\geq T_1\),
		\begin{align}
			\Phi(t,\tau-t,\theta_{-t}\omega,D(\tau-t,\theta_{-t}\omega))=u(\tau,\tau-t,\theta_{-\tau}\omega,D(\tau-t,\theta_{-t}\omega))\subseteq K_V(\tau,\omega).\notag
		\end{align}
		Then we show that \(K_V\in\mathcal{D}_V\).
		Let \(c\) be an arbitrary positive constant, we have
		\begin{align}
			&\lim\limits_{t\to-\infty}e^{ct}|K_V(\tau+t,\theta_t\omega)|_{V}^2=\lim\limits_{t\to-\infty}e^{ct}R_V(\tau+t,\theta_t\omega)\notag\\
			=&M_1\left[1 +\lim\limits_{t\to-\infty}e^{ct}|z(\theta_t\omega)|^2 +\lim\limits_{t\to-\infty}e^{ct}\int_{-\infty}^{0}e^{\left(\nu\lambda_1-\frac{2M_\epsilon^2}{\nu}\right)r}\left(\|f(r+\tau+t,\cdot)\|_{L^2}^2+|z(\theta_{r+t}\omega)|^2\right)dr\right].\notag
		\end{align}
		By \eqref{z polynomial growth}, we obtain
		\begin{align}
			\lim\limits_{t\to-\infty}e^{ct}|z(\theta_t\omega)|^2=0,\notag\\
			\lim\limits_{t\to-\infty}e^{ct}\int_{-\infty}^{0}e^{\left(\nu\lambda_1-\frac{2M_\epsilon^2}{\nu}\right)r}|z(\theta_{r+t}\omega)|^2dr =0.\notag
		\end{align}
		By (A2), we have
		\begin{align}
			\lim\limits_{t\to-\infty}e^{ct}\int_{-\infty}^{0}e^{\left(\nu\lambda_1-\frac{2M_\epsilon^2}{\nu}\right)r}\|f(r+\tau+t,\cdot)\|_{L^2}^2dr=0.\notag
		\end{align}
		Hence, we get \(\lim\limits_{t\to-\infty}e^{ct}|K_V(\tau+t,\theta_t\omega)|_V^2=0\) for all \(c>0\), and this implies \(K_V\in\mathcal{D}_V\).
	\end{proof}
\end{lem}

\subsection{Pullback flattening}
We proceed to prove that the cocycle \(\Phi\) is $\mathcal{D}_V$-pullback flattening. The key estimates for this property are derived similarly to those proving the existence of a $\mathcal{D}_V$-pullback absorbing set.
\begin{lem}\label{lem:flattening property}
		Let $V_\delta={\rm span}\{w_1, w_2,...,w_{N_\delta}\}$ be a finite-dimensional subspace of \(V\) spanned by the first \(N_\delta(\in\mathbb{N})\) eigenfunctions $w_i\ (i=1,\ldots,N_\delta)$ of Stokes operator \(A\).
		Let \(P_\delta: V\to V_\delta\) be a bounded projection. Then for every \(\tau\in\mathbb{R}\), \(\omega\in\Omega\), \(\delta>0\) and \(D=\{D(\tau,\omega):\tau\in\mathbb{R},\omega\in\Omega\}\in\mathcal{D}_V\), there exists \(T_1=T_1(\tau,\omega,D)>0\) such that for all \(t\geq T_1\),
		\item[(i)] \(\|P_\delta u(\tau,\tau-t,\theta_{-\tau}\omega,u_{\tau-t})\|_{V}\) is bounded, and
		\item[(ii)] \(\|(I-P_{\delta})u(\tau,\tau-t,\theta_{-\tau}\omega,u_{\tau-t})\|_{V}<\delta\),
		where \(u_{\tau-t}=v_{\tau-t}+h(x)z(\theta_{-t}\omega)\), \(v_{\tau-t}\in D(\tau-t,\theta_{-t}\omega)\).
\end{lem}
\begin{proof}
	The first boundedness condition (i) clearly holds from Lemma~\ref{lem:absorbing set} and the fact that \(\|P_\delta u\|_{V}\leq\|u\|_{V}\) for all \(u\in V\). Then we verify the second condition (ii). Let \(Q=I-P_\delta\) and
	by \(u(\tau,\tau-t,\theta_{-\tau}\omega,u_{\tau-t})=v(\tau,\tau-t,\theta_{-\tau}\omega,v_{\tau-t})+h(x)z(\omega)\), we have
	\begin{align}
		\|Qu(\tau,\tau-t,\theta_{-\tau}\omega,u_{\tau-t})\|_{V}\leq
		\|Qv(\tau,\tau-t,\theta_{-\tau}\omega,v_{\tau-t})\|_{V} +|z(\omega)|\|Qh(x)\|_{V}.\notag
	\end{align}
	Since \(h\in \mathcal{D}(A)\), we obtain
	\begin{align}
		|z(\omega)|\|Qh(x)\|_{V}<\frac{1}{\sqrt{\lambda_{N_\delta}}}|z(\omega)\|Q\Delta h\|_{L^2}<\frac{\delta}{2},\notag
	\end{align}
	as \(\lambda_{N_\delta}\to +\infty\).
	Taking the action of \eqref{weak formulation} with \(AQv\), we obtain
	\begin{align}
		\frac{d}{dt}\|\nabla Qv\|_{L^2}^2 +\nu\|AQv\|_{L^2}^2
		\leq
		\frac{2M_\epsilon^2}{\nu}\|\nabla v\|_{L^2}^2 +C_{\nu,\epsilon}\left(\|f\|_{L^2}^2 +|z(\theta_t\omega)|^2\right). \notag
	\end{align}
	Since \(\lambda_{N_\delta}\|\nabla Qv\|_{L^2}^2\leq \|AQv\|_{L^2}^2\), we have
	\begin{align}
		\frac{d}{dt}\|\nabla Qv\|_{L^2}^2 +\nu\lambda_{N_\delta}\|\nabla Qv\|_{L^2}^2
		\leq
		\frac{2M_\epsilon^2}{\nu}\|\nabla v\|_{L^2}^2 +C_{\nu,\epsilon}\left(\|f\|_{L^2}^2 +|z(\theta_t\omega)|^2\right). \notag
	\end{align}
	Multiplying \(e^{\nu\lambda_{N_\delta}t}\) and integrating on \([\tau-t,\tau]\), we obtain
	\begin{align}
		&\|\nabla Qv(\tau,\tau-t,\omega,v_{\tau-t})\|_{L^2}^2 \notag\\
		\leq
		&\ e^{-\nu\lambda_{N_\delta}t}\|\nabla Qv_{\tau-t}\|_{L^2}^2 +\frac{2M_\epsilon^2}{\nu}e^{-\nu\lambda_{N_\delta}\tau}\int_{\tau-t}^{\tau}e^{\nu\lambda_{N_\delta}s}\|\nabla v(s,\tau-t,\omega,v_{\tau-t})\|_{L^2}^2ds \notag\\
		&\ +
		C_{\nu,\epsilon}e^{-\nu\lambda_{N_\delta}\tau}\int_{\tau-t}^{\tau}e^{\nu\lambda_{N_\delta}s}\left(\|f(s,\cdot)\|_{L^2}^2+|z(\theta_s\omega)|^2\right)ds.\notag
	\end{align}
	Replacing \(\omega\) with \(\theta_{-\tau}\omega\), we have
	\begin{align}
		&\|\nabla Qv(\tau,\tau-t,\theta_{-\tau}\omega,v_{\tau-t})\|_{L^2}^2\notag\\
		\leq
		&\
		e^{-\nu\lambda_{N_\delta}t}\|\nabla Qv_{\tau-t}\|_{L^2}^2 +\frac{2M_\epsilon^2}{\nu}e^{-\nu\lambda_{N_\delta}\tau}\int_{\tau-t}^{\tau}e^{\nu\lambda_{N_\delta}s}\|\nabla v(s,\tau-t,\theta_{-\tau}\omega,v_{\tau-t})\|_{L^2}^2ds \notag\\
		&\ +
		C_{\tau,\epsilon}e^{-\nu\lambda_{N_\delta}\tau}\int_{\tau-t}^{\tau}e^{\nu\lambda_{N_\delta}s}\left(\|f(s,\cdot)\|_{L^2}^2+|z(\theta_{s-\tau}\omega)|^2\right)ds\notag\\
		\leq
		&\
		e^{-\nu\lambda_{N_\delta}t}\|\nabla v_{\tau-t}\|_{L^2}^2 +\frac{2M_\epsilon^2}{\nu}e^{-\nu\lambda_{N_\delta}\tau}\int_{\tau-t}^{\tau}e^{\nu\lambda_{N_\delta}s}\|\nabla v(s,\tau-t,\theta_{-\tau}\omega,v_{\tau-t})\|_{L^2}^2ds \notag\\
		&\ +
		C_{\nu,\epsilon}e^{-\nu\lambda_{N_\delta}\tau}\int_{-\infty}^{\tau}e^{\nu\lambda_{N_\delta}s}\|f(s,\cdot)\|_{L^2}^2ds
		+C_{\nu,\epsilon}
		e^{-\nu\lambda_{N_\delta}\tau}\int_{-t}^{0}e^{\nu\lambda_{N_\delta}(s+\tau)}|z(\theta_{s}\omega)|^2ds\notag\\
		:= &\ I_1+I_2+I_3+I_4. \notag
	\end{align}
	It follows from \(v_{\tau-t}\in D(\tau-t,\theta_{-t}\omega)\) that
	\begin{align}\label{I1to0}
		I_1\leq e^{-\nu\lambda_{N_\delta}t}\|D(\tau-t,\theta_{-t}\omega)\|_V^2\to0\ \ \ \text{as}\ t\to\infty.
	\end{align}
	By Lemma~\ref{lem:uniform estimates}, we obtain
	\be[I2to0]
		I_2\leq
		&C_{\nu,\epsilon}e^{-\nu\lambda_{N_\delta}\tau}\int_{\tau-t}^{\tau}e^{\nu\lambda_{N_\delta}s}e^{-\left(\nu\lambda_1-\frac{2M_\epsilon^2}{\nu}\right)\left(s-\tau\right)}\left[1 +\int_{-\infty}^{s-\tau}e^{\left(\nu\lambda_1-\frac{2M_\epsilon^2}{\nu}\right)r}\|f(r+\tau,\cdot)\|_{L^2}^2dr \vphantom{\int_{-\infty}^{s-\tau}}\right.\notag\\
		&\
		\left.+\int_{-\infty}^{s-\tau}e^{\left(\nu\lambda_1-\frac{2M_\epsilon^2}{\nu}\right)r}|z(\theta_r\omega)|^2dr \vphantom{\int_{-\infty}^{s-\tau}}\right]ds\notag\\
		\leq
		&e^{-\nu\lambda_{N_\delta}\tau}\int_{\tau-t}^{\tau}e^{\nu\lambda_{N_\delta}s}e^{-\left(\nu\lambda_1-\frac{2M_\epsilon^2}{\nu}\right)\left(s-\tau\right)}C(\tau,\omega,\epsilon)ds\notag\\
		\leq&
		\ C(\tau,\omega,\epsilon) e^{-\left(\nu\lambda_{N_\delta}-\nu\lambda_1+\frac{2M_\epsilon^2}{\nu}\right)\tau}\int_{\tau-t}^{\tau}e^{\left(\nu\lambda_{N_\delta}-\nu\lambda_1+\frac{2M_\epsilon^2}{\nu}\right)s}ds\notag\\
		\leq
		&\frac{1}{\nu\lambda_{N_\delta}-\nu\lambda_1+\frac{2M_\epsilon^2}{\nu}}C(\tau,\omega,\epsilon)\notag\\
		\to &\ 0 ,
	\ee
	as \(\lambda_{N_\delta}\to +\infty\), where
	\begin{align}
		C(\tau,\omega,\epsilon)=C_{\nu,\epsilon}\left(1 +\int_{-\infty}^{0}e^{\left(\nu\lambda_1-\frac{2M_\epsilon^2}{\nu}\right)r}\|f(r+\tau,\cdot)\|_{L^2}^2dr
		+\int_{-\infty}^{0}e^{\left(\nu\lambda_1-\frac{2M_\epsilon^2}{\nu}\right)r}|z(\theta_r\omega)|^2dr\right)\notag
	\end{align} is bounded.
	It follows from (\cite{kloeden2007pullback}, Lemma 12) that
	\begin{align}\label{I3to0}
		\lim\limits_{\lambda_{N_\delta}\to+\infty}I_3=0.
	\end{align}
	For \(I_4\), by setting \(\alpha\in(0,\nu\lambda_1)\) and using \eqref{z polynomial growth}, we have
	\be[I4to0]
		& e^{-\nu\lambda_{N_\delta}\tau}\int_{-t}^{0}e^{\nu\lambda_{N_\delta}(s+\tau)}|z(\theta_{s}\omega)|^2ds
		\notag\\
		=&\int_{-t}^{0}e^{\left(\nu\lambda_{N_\delta}-\frac{\alpha}{2}\right)s} e^{\frac{\alpha}{2}s}|z(\theta_s\omega)|^2ds\notag\\
		\leq
		&
		\left(\int_{-t}^{0}e^{2\left(\nu\lambda_{N_\delta}-\frac{\alpha}{2}\right)s}ds\right)^\frac{1}{2}
		\left(\int_{-t}^{0}e^{\alpha s}|z(\theta_s\omega)|^4ds\right)^\frac{1}{2}\notag\\
		\leq
		& \left(\frac{1-e^{-2\left(\nu\lambda_{N_\delta}-\frac{\alpha}{2}\right)t}}{2\left(\nu\lambda_{N_{\delta}}-\frac{\alpha}{2}\right)}\right)^\frac{1}{2}
		\left(\int_{-\infty}^{0}e^{\alpha s}|z(\theta_s\omega)|^4ds\right)^\frac{1}{2}\notag\\
		\to &\ 0,
	\ee
	as \(\lambda_{N_\delta}\to +\infty\).
	Thus, by \eqref{I1to0}, \eqref{I2to0}, \eqref{I3to0} and \eqref{I4to0}, there exists \(T_1=T_1(\tau,\omega,D)>0\), for all \(t\geq T_1\) and \(N_{\delta}\to\infty\),
	\begin{align}
		\|\nabla Qv(\tau,\tau-t,\theta_{-\tau}\omega,v_{\tau-t})\|_{L^2}<\frac{\delta}{2}.\notag
	\end{align}
	The proof completes.
\end{proof}

\subsection{Existence of a $\mathcal{D}_V$-pullback attractor}
Our main theorem concerns the existence of $\mathcal{D}_V$-pullback attractors for $\Phi$.
\begin{thm}
	Suppose \(f\in L_{\rm loc}^2(\mathbb{R};H)\), assumptions (A1), (A2) and (A3) hold. Then the continuous cocycle $\Phi$ associated with problem \eqref{calmed u} possesses a unique $\mathcal{D}_V$-pullback attractor $\mathcal{A}=\left\{\mathcal{A}(\tau,\omega):\tau\in\mathbb{R},\omega\in\Omega)\right\}\in\mathcal{D}_V$ in $V$.
\end{thm}
\begin{proof}
	By Lemma~\ref{lem:absorbing set}, $\Phi$ admins a closed measurable $\mathcal{D}_V$-pullback absorbing set in $V$. Furthermore, Lemma~\ref{lem:flattening property} guaranteers that $\Phi$ is $\mathcal{D}_V$-pullback flattening in $V$. Based on these results and invoking Proposition~\ref{prop:pullback attractor}, we deduce that $\Phi$ has a unique $\mathcal{D}_V$-pullback attractor $\mathcal{A}$ in $V$.
\end{proof}

\bibliographystyle{abbrv}
\bibliography{reference11}

\begin{thebibliography}{10}

\bibitem{Caraballo2023Random}
T.~Caraballo, Z.~Chen, and D.~Yang.
\newblock Random dynamics and limiting behaviors for 3{D} globally modified
  {N}avier-{S}tokes equations driven by colored noise.
\newblock {\em Studies in Applied Mathematics}, 151(1):247--284, 2023.

\bibitem{caraballo2013three}
T.~Caraballo and P.~E. Kloeden.
\newblock The three-dimensional globally modified {Navier}--{Stokes} equations:
  recent developments.
\newblock In {\em Recent Trends in Dynamical Systems: Proceedings of a
  Conference in Honor of J{\"u}rgen Scheurle}, pages 473--492. Springer Basel,
  2013.

\bibitem{2006Unique}
T.~Caraballo, J.~Real, and P.~E. Kloeden.
\newblock Unique strong solutions and {V}-attractors of a three dimensional
  system of globally modified {Navier}-{Stokes} equations.
\newblock {\em Advanced Nonlinear Studies}, 6(3):411--436, 2006.

\bibitem{constantin1988navier}
P.~Constantin and C.~Foia{\c{s}}.
\newblock {\em Navier-{S}tokes {Equations}}.
\newblock University of Chicago press, 1988.

\bibitem{dauge1989stationary}
M.~Dauge.
\newblock Stationary {Stokes} and {Navier}--{Stokes} systems on two or
  three-dimensional domains with corners. {P}art {I}. {L}inearized equations.
\newblock {\em SIAM Journal on Mathematical Analysis}, 20(1):74--97, 1989.

\bibitem{enlow2024calmedOH}
M.~Enlow, A.~Larios, and Y.~Pei.
\newblock Calmed {Ohmic} heating for the 2{D}
  {Magnetohydrodynamic}-{Boussinesq} system: Global well-posedness and
  convergence, 2024.

\bibitem{enlow2024algebraic}
M.~Enlow, A.~Larios, and J.~Wu.
\newblock Algebraic calming for the 2{D} {Kuramoto}-{Sivashinsky} equations.
\newblock {\em Nonlinearity}, 37(11):115019, 2024.

\bibitem{enlow2024calmedNS}
M.~Enlow, A.~Larios, and J.~Wu.
\newblock Calmed 3{D} {Navier}--{Stokes} equations: Global well-posedness,
  energy identities, global attractors, and convergence.
\newblock {\em Journal of Nonlinear Science}, 34(6):1--33, 2024.

\bibitem{guermond2011note}
J.-L. Guermond and A.~Salgado.
\newblock A note on the {Stokes} operator and its powers.
\newblock {\em Journal of Applied Mathematics and Computing}, 36:241--250,
  2011.

\bibitem{hang2024random}
H.~T. Hang and P.~T. Nguyen.
\newblock Random attractors for three-dimensional stochastic globally modified
  {Navier}--{Stokes} equations driven by additive noise on unbounded domains.
\newblock {\em Random Operators and Stochastic Equations}, 32(3):223--239,
  2024.

\bibitem{kloeden2007flattening}
P.~E. Kloeden and J.~A. Langa.
\newblock Flattening, squeezing and the existence of random attractors.
\newblock {\em Proceedings of the Royal Society A: Mathematical, Physical and
  Engineering Sciences}, 463(2077):163--181, 2007.

\bibitem{kloeden2007pullback}
P.~E. Kloeden, J.~A. Langa, and J.~Real.
\newblock Pullback {V}-attractors of the three dimensional globally modified
  {Navier}-{Stokes} equations.
\newblock {\em Communications on Pure and Applied Analysis}, 6(4):937--955,
  2007.

\bibitem{robinson2001infinite}
J.~C. Robinson.
\newblock {\em Infinite-{Dimensional} {Dynamical} {Systems}: {An}
  {Introduction} to {Dissipative} {Parabolic} PDEs and the {Theory} of {Global}
  {Attractors}}, volume~28.
\newblock Cambridge University Press, 2001.

\bibitem{romito2009uniqueness}
M.~Romito.
\newblock The uniqueness of weak solutions of the globally modified
  {Navier}-{Stokes} equations.
\newblock {\em Advanced Nonlinear Studies}, 9(2):425--427, 2009.

\bibitem{Tao2006NonlinearDE}
T.~Tao.
\newblock {\em Nonlinear {Dispersive} {Equations}: {Local} and {Global}
  {Analysis}}.
\newblock Number 106 in Conference Board of the Mathematical Sciences Regional
  Conference Series in Mathematics. American Mathematical Society, Providence,
  RI, 2006.

\bibitem{temam2001navier}
R.~Temam.
\newblock {\em Navier--{Stokes} {Equations}: {Theory} and {Numerical}
  {Analysis}}, volume 343.
\newblock American Mathematical Society, 2001.

\bibitem{wang2012sufficient}
B.~Wang.
\newblock Sufficient and necessary criteria for existence of pullback
  attractors for non-compact random dynamical systems.
\newblock {\em Journal of Differential Equations}, 253(5):1544--1583, 2012.

\bibitem{wang1993remark}
X.~Wang.
\newblock A remark on the characterization of the gradient of a distribution.
\newblock {\em Applicable Analysis}, 51(1-4):35--40, 1993.

\end{thebibliography}

\end{document}